\documentclass[12pt,a4paper,reqno]{amsart}
\usepackage[mathcal]{eucal}
\usepackage{amssymb}
\usepackage[nice]{nicefrac}
\usepackage{hyperref}
\usepackage[normalem]{ulem}
\usepackage{xcolor}

\usepackage{amsthm}
\usepackage{mathtools}
\usepackage{amsfonts}
\usepackage{comment}
\usepackage{faktor}

\theoremstyle{plain}
\newtheorem{theorem}{Theorem}[section]
\newtheorem{lemma}[theorem]{Lemma}
\newtheorem{proposition}[theorem]{Proposition}

\theoremstyle{definition}

\numberwithin{equation}{section}

\makeatletter
\newcommand{\AlignFootnote}[1]{%
    \ifmeasuring@
    \else
        \footnote{#1}%
    \fi
}
\makeatother

\title{Finite and profinite groups with small Engel sinks of $p$-elements}

\author[L. Dal Berto]{Lucas Dal Berto} 
\address{Lucas Dal Berto: Department of Mathematics, University of Brasilia, Brasilia, DF, Brazil}
\email{lucasmatheusdelimadalberto@gmail.com}

\author[J.Caldeira]{Jhone Caldeira} 
\address{Jhone Caldeira: Institute of Mathematics and Statistics, Federal University of Goiás, Goiânia, GO, Brazil}
\email{jhone@ufg.br}

\author[P. Shumyatsky]{Pavel Shumyatsky} 
\address{Pavel Shumyatsky: Department of Mathematics, University of Brasilia, Brasilia, DF, Brazil}
\email{pavel@unb.br}

\thanks{The work of the first author was supported by CAPES.\\\indent The work of the third author was supported by FAPDF and CNPq}


\usepackage{mathrsfs}
\newtheorem*{theorem*}{Theorem \ref{main}}
\newtheorem*{theorem**}{Theorem \ref{h_lim}}

\begin{document}

\maketitle

\begin{abstract}
A (left) Engel sink of an element $g$ of a group $G$ is a subset containing all sufficiently long commutators $[...[[x,g],g],\dots,g]$, where $x$ ranges over $G$. We prove that if $p$ is a prime and $G$ a finite group in which, for some positive integer $m$, every $p$-element has an Engel sink of cardinality at most $m$, then $G$ has a normal subgroup $N$ such that $G/N$ is a $p'$-group and the index $[N:O_p(G)]$ is bounded in terms of $m$ only.

\noindent Furthermore, if $G$ is a profinite group in which every $p$-element possesses a finite Engel sink, then $G$ has a normal subgroup $N$ such that $N$ is virtually pro-$p$ while $G/N$ is a pro-$p'$ group.
 \end{abstract}

\section{Introduction}

Recall that a group $G$ is Engel if for any elements $x,y\in G$ there is a positive number $n$ such that $[x,{}_ny]=1$. Throughout, we use the left-normed commutator notation $[x,{}_ny]=[x,y,\dots,y]$, where $y$ is repeated $n$ times. More generally, an element $y\in G$ is called (left) Engel if for any $x\in G$ there is a number $n$ such that $[x,{}_ny]=1$. Finite Engel groups are nilpotent by Zorn’s theorem \cite[12.3.4]{robinson}. Moreover, by Baer’s theorem \cite[12.3.7]{robinson} a left Engel element of a finite group belongs to the Fitting subgroup.

An interesting generalisation of the Engel condition was considered in several recent papers \cite{ksijac,ks_almost,ks_engel}. A (left) Engel sink of an element $g\in G$ is a set $\mathscr{E} (g)$ such that for every $x\in G$ all sufficiently long commutators $[x,g,g,\dots,g]$ belong to $\mathscr{E}(g)$, that is, for every $x\in G$ there is a positive integer $l(x,g)$ such that $[x,{}_lg]\in\mathscr{E}(g)$ for all $l\geq l(x, g)$.

When $G$ is a finite group, then every element has the smallest (left) Engel sink, since the intersection of two Engel sinks $\mathscr{E}(g)$ and $\mathscr{E}''(g)$ is again an Engel sink of $g$. Throughout this paper the symbol $\mathscr{E}(g)$ stands for the minimal Engel sink of $g$. It was shown in \cite{ksijac} that if $G$ is a finite group in which $\mathscr{E}(g)$ has cardinality at most $m$ for every $g\in G$, then $G$ has a normal subgroup $N$ such that $G/N$ is nilpotent and the order of $N$ is bounded in terms of $m$ only. This is a natural extension of the aforementioned Zorn's theorem.

In the present paper we consider finite groups $G$ such that, for a fixed prime $p$, all $p$-elements of $G$ have sinks of size at most $m$. Obviously, if $\mathscr{E}(g)=1$ for every $p$-element $g\in G$, then $G/O_p(G)$ is a $p'$-group. Here, as usual, $O_p(G)$ stands for the largest normal $p$-subgroup of a finite group $G$. It is clear that if $m\geq2$ is the minimum of cardinalities of $\mathscr{E}(g)$, then necessarily $m\equiv1$ modulo $p$. This is because $\langle g\rangle$ acts by conjugation on the set $\mathscr{E}(g)\setminus\{1\}$ and there are no fixed points under this action. So, in particular, $m\neq2$. On the other hand, for any odd $m$ there is a group in which $|\mathscr{E}(g)|\leq m$ for every $2$-element $g\in G$ (consider the dihedral group of order $2m$).

The purpose of this paper is to prove the following theorem.

\begin{theorem}\label{main1} Let $p$ be a prime and $G$ a finite group such that $|\mathscr{E}(g)| \leq m$ for all $p$-elements $g\in G$. Then $G$ has a normal subgroup $N$ such that $G/N$ is a $p'$-group and the index $[N:O_p(G)]$ is bounded in terms of $m$ only.
\end{theorem} 

Theorem \ref{main1} can be naturally extended to profinite groups $G$ in which any $p$-element has a finite Engel sink of cardinality at most $m$.
On the other hand, in the realm of profinite groups, it is more natural to handle the setting where the cardinalities of sinks are finite but not necessarily bounded.
 
In this paper we prove the following profinite variation of Theorem \ref{main1}.

\begin{theorem}\label{main2} Let $p$ be a prime and $G$ a profinite group in which every $p$-element possesses a finite Engel sink. Then $G$ has a  normal subgroup $N$ such that $N$ is virtually pro-$p$ while $G/N$ is a pro-$p'$ group.
\end{theorem} 

Recall that a profinite group is said to possess a property virtually if it has an open subgroup with that property.
Somewhat surprisingly, the proof of Theorem \ref{main2} given in this paper does not use Theorem \ref{main1} at all.

\section{Proof of Theorem \ref{main1}}
Throughout the paper, we say that a quantity is $(a,b,c\dots)$-bounded if it is bounded in terms of the parameters $a,b,c\dots$.
We start with a lemma which lists some well-known properties of coprime actions (see for example \cite[Ch.~5 and 6]{go}).  In the sequel the lemma will often be used without explicit references. As usual, $[G,A]$ denotes the subgroup generated by commutators $x^{-1}x^a$, where $x\in G$ and $a\in A$. The symbol $\langle X\rangle$ stands for the subgroup generated by the set $X$.

\begin{lemma}\label{cc}
Let  $A$ be a group of automorphisms of a finite group $G$ such that $(|A|,|G|)=1$. Then
\begin{enumerate}
\item[(i)] $G=[G,A]C_{G}(A)$. If $G$ is abelian, then $G=[G,A]\oplus C_{G}(A)$.
\item[(ii)] $[G,A,A]=[G,A]$. 
\item[(iii)] $C_{G/N}(A)=NC_G(A)/N$ for any $A$-invariant normal subgroup $N$ of $G$.
\item[(iv)] If $[G/\Phi(G),A]=1$, then $[G,A]=1$.
\end{enumerate}
\end{lemma}

According to \cite[Lemma 2.1]{ks_almost}, the Engel sink $\mathscr{E}(g)$ of an element $g\in G$ consists precisely of all elements $x$ such that $x = [x,g,g,\ldots,g]$, where $g$ occurs at least once. The following lemma is straightforward.

\begin{lemma}\label{E=[E,g]}
Let $g$ be an element of a finite group $G$ and write $E=\langle\mathscr{E}(g)\rangle$. Then $E = [E,g]$.
\end{lemma}

\begin{lemma}\label{N_Engel}
Let $g$ be an element of a finite group $G$, and let $N$ be a normal subgroup of $G$. Then $N\cap\mathscr{E}(g)=1$ if and only if there exists a positive integer $k$ such that $[N,{}_kg]=1$.
\end{lemma}

\begin{proof}
Suppose that $N \cap \mathscr{E}(g) = 1$. If an element $x\in N$ can be written as $x = [x,{}_ig]$ for some $i\geq1$, then clearly $x=1$. It follows that there is a positive integer $k$ such that $[N,{}_kg]=1$.
	
On the other hand, if there exists a positive integer $k$ such that $[N,{}_kg]=1$, then no nontrivial element $x\in N$ satisfies an equality $x = [x,{}_ig]$. Therefore we conclude that $N \cap \mathscr{E}(g) = 1$.
\end{proof}

Throughout, we write $F(G)$ to denote the Fitting subgroup of $G$.

\begin{lemma}\label{corr} For any element $g$ of  a finite group $G$ there is a unique maximal normal subgroup $N$ of $G$ such that $N\cap\mathscr{E}(g)=1$. We have $[N,g]\leq F(G)$.
\end{lemma}

\begin{proof}  To prove the uniqueness of $N$ it is sufficient to show that if $N_1$ and $N_2$ are normal subgroups of $G$ such that $N_1\cap\mathscr{E}(g)=1=N_2\cap\mathscr{E}(g)$, then $N_1N_2\cap\mathscr{E}(g)=1$.

Since $N_1 \cap \mathscr{E}(g) = 1 = N_2 \cap \mathscr{E}(g)$, by Lemma \ref{N_Engel} there exist $k_1, k_2$ such that $[N_1,{}_{k_1}g]=1=[N_2,{}_{k_2}g]$. It follows that $[N_1N_2,{}_{k_2}g]\leq N_1$ and therefore $[N_1N_2,{}_{k_1+k_2}g]=1$. Because of Lemma \ref{N_Engel}, this implies that $N_1 N_2 \cap \mathscr{E}(g)=1$.

To see that $[N,g]\leq F(G)$, observe that $g$ is a left Engel element of $N\langle g\rangle$. By \cite[Satz III.6.15]{baer_theo}, we have $g\in F(N\langle g\rangle)$. Therefore, $[N,g]$ is nilpotent and so $[N,g]\leq F(G)$.
\end{proof}

Recall that the Fitting series of a finite soluble group $G$ is defined starting with $F_0(G) = 1$, and then inductively, $F_{k+1}(G)$ is the full preimage of $F(G/F_k(G))$. The smallest number $h=h(G)$ such that $F_h(G) = G$ is called the Fitting height of $G$. 

\begin{lemma}\label{h_lim} Let $G$ be a finite soluble group and $g\in G$ such that $|\mathscr{E}(g)|\leq m$. Then $h(\langle \mathscr{E}(g)\rangle)\leq2m-2$.
\end{lemma}

\begin{proof} We use induction on $m$. If $m=1$, the claim is correct so we assume that $m\geq2$. Set $E=\langle \mathscr{E}(g)\rangle$. Since $\mathscr{E}(g)$ consists precisely of all elements $x$ such that $x = [x,g,g,\ldots,g]$, where $g$ occurs at least once, without loss of generality we can assume that $G= E\langle g\rangle$. Let $N$ be the maximal normal subgroup of $G$ such that $N\cap\mathscr{E}(g)=1$. In view of Lemma \ref{corr} $[N,g]\leq F(G)$. Set $\overline{G}=G/F(G)$. It follows that $[\overline{N},\overline{g}]=1$. Since $E=[E,g]$, we conclude that $\overline{N}\leq Z(\overline{G})$. 
	
Observe that in $G/N$ any nontrivial normal subgroup has nontrivial intersection with $\mathscr{E}(gN/N)$. Let $M/N$ be a minimal normal subgroup of $G/N$. Then $|\mathscr{E}(gM/M)|\leq m-1$. 
	
By induction, $h(EM/M)\leq 2m-4$. Since $M/N$ is elementary abelian and $\overline{N}\leq Z(\overline{G})$, it follows that $M\leq F_2(G)$. Therefore $h(E)\leq2m-2$.\end{proof}

We will require the following two lemmas, which are immediate from  \cite[Lemma 3.2]{ks_almost} and \cite[Lemma 3.3]{ks_almost}, respectively.

\begin{lemma}\label{[N,g]_m-lim} Let $p$ be a prime and $N$ a nilpotent normal $p'$-subgroup of a finite group $G$. Then, for any $p$-element $g\in G$, the order of $[N,g]$ is bounded in terms of the size of the Engel sink $\mathscr{E}(g)$.
\end{lemma}

\begin{lemma}\label{[V,U]_m-lim} Let $V$ be an abelian $p'$-group and $U$ a $p$-group of automorphisms of $V$. If $|[V,u]|\leq m$ for every $u\in U$, then the orders of $[V,U]$ and $U$ are both $m$-bounded.
\end{lemma}

Given a prime $p$, we write $O_{p'}(G)$ for the largest normal $p'$-subgroup of a finite group $G$. 
\begin{proposition}\label{[N,P]_m-lim}
Let $G=NP$ be a finite soluble group, where $P$ is a Sylow $p$-subgroup and $N = O_{p'}(G)$. Suppose that $|\mathscr{E}(g)| \leq m$ for all $g\in P$. Then the order of $[N,P]$ is $m$-bounded.
\end{proposition}

\begin{proof} Without loss of generality, we can assume that $O_p(G) = 1$ and $N=[N,P]$. We will first show that the Fitting height $h(N)$ of $N$ is $m$-bounded. For an arbitrary element $g\in P$ set $h_g=h(\langle\mathscr{E}(g)\rangle)$. In view of Lemma \ref{h_lim}  the numbers $h_g$ are bounded in terms of $m$ only. Set $$h_0 = \max\{h_g\mid g\in P\}.$$ According to \cite[Theorem 1.1]{ks_engel} for any $g\in P$ we have $g\in F_{h_0+1}(G)$. It follows that $[N,g] \subseteq F_{h_0+1}(G)$. Since this holds for any $g\in P$, we conclude that $[N,P]\subseteq F_{h_0+1}(G)$ whence $h(N) \leq h_0+1$, as required.
	
The proposition will be proved using induction on the Fitting height of $N$.
	
Suppose first that $h(N)=1$, that is, $N$ is nilpotent. The subgroup $P$ faithfully acts on $V=N/\Phi(N)$ and for any $g\in P$ we have $|[V,g]|\leq m$. In view of Lemma \ref{[V,U]_m-lim} we conclude that $|P|$ is $m$-bounded.
	
Note that $N$ is the product of subgroups of the form $[N,g]$ for $g\in P$. By Lemma \ref{[N,g]_m-lim} the order of any subgroup $[N,g]$ is $m$-bounded and since $|P|$ is $m$-bounded, it follows that the order of $N$ is $m$-bounded, too. This completes the case where $N$ is nilpotent. 
	
Assume now that $h(N)\geq2$ and let $M = F(N)$ be the Fitting subgroup of $N$. From the previous case we know that the order of $[M,P]$ is $m$-bounded. Moreover, by induction, the index of $M$ in $N$ is also $m$-bounded. Since $M$ normalizes $[M,P]$, it follows that $N_G([M,P])$ has $m$-bounded index. Set $X = \prod_{x\in G} [M,P]^x$ and observe that the order of $X$ is $m$-bounded. 
	
Set $\bar{G}=G/X$ and note that $\bar{P}$ centralizes $\bar{M}$. Therefore $\bar{N}=[\bar{N},\bar{P}]$ centralizes $\bar{M}$ and $\bar{M}\leq Z(\bar{N})$. It follows that $h(\bar{N})\leq h-1$. By induction, the oder of $\bar{N}$ is $m$-bounded. Since we have already shown that $|X|$ is $m$-bounded, we conclude that $|N|$ is $m$-bounded. This completes the proof.
\end{proof}

Whenever $a$ is an automorphism of a group $G$, we write $\mathcal{E}_G(a)$ to denote $\mathcal{E}(a)$ in the group $G\langle a\rangle$. We say that a group is semisimple if it is isomorphic to a direct product of non-abelian simple groups.

\begin{lemma}\label{G_semis_lim} Let $G$ be a semisimple group admitting an automorphism $\alpha$ such that $G=[G,\alpha]$ and $|\mathcal{E}_G(\alpha)| \leq m$. Then the order of $G$ is $m$-bounded. Consequently, the order of $\alpha$ is $m$-bounded, too.
\end{lemma}

\begin{proof} Let $C = C_G(\alpha)$. Obviously $C$ normalizes the set $\mathcal{E}_G(\alpha)$. It follows that $|C / C_C(\mathcal{E}_G(\alpha))| \leq m!$.
	
By a theorem of Guralnick and Tracey \cite[Theorem 1.4]{guratra} the group $G$ is generated by $\mathcal{E}_G(\alpha)$. Since $G$ is a semisimple group, it follows that $C_G(\mathcal{E}_G(\alpha))=1$. Therefore, the order of $C$ is at most $m!$. The same argument shows  that the order of $\alpha$ is at most $m!$.
	
A well-known theorem of Hartley \cite[Theorem A]{hartley} now tells us that $G$ contains a soluble subgroup of $m$-bounded index. Since $G$ is semisimple,  the result follows.
\end{proof}

Let $G$ be a finite semisimple group and write $G= G_1 \times G_2 \times \cdots \times G_k$, where each $G_i$ is a non-abelian simple group. For an automorphism $\alpha$ of $G$ write $$C_0(\alpha) = G_{i_1} \times  \cdots \times G_{i_t}$$ to denote the product of all simple factors $G_i$ contained in $C_G(\alpha)$. Obviously we have $G = [G,\alpha] \times C_0(\alpha)$, where $[G,\alpha]$ is the product of the other $k-t$ simple factors.

Suppose there is another automorphism $\beta \in Aut(G)$ such that $[\alpha, \beta] = 1$ and $[G, \alpha, \beta] = 1$. We conclude that $[G,\alpha]\leq C_0(\beta)$ and $[G,\beta]\leq C_0(\alpha)$. Hence, $[G,\alpha]\cap[G,\beta]=1$ and $[G,\beta,\alpha]=1$. Moreover, $G=C_{0}(\alpha)\times C_{0}(\beta)$.

We therefore have
$$G_i^{\alpha\beta} = \left\lbrace \begin{array}{lc}
	G_i^{\alpha} & \text{, if } G_i \subseteq C_{0}(\beta)  \\
	G_i^{\beta} & \text{, if } G_i \subseteq C_{0}(\alpha)
\end{array} \right..$$
It follows that $[G, \beta][G, \alpha] \leq [G, \alpha\beta]$. 

In the sequel we will require the fact that if $G$ is a finite group in which every abelian subgroup has order at most $m$, then the order of $G$ is $m$-bounded (cf. \cite{ai,ag}).

\begin{lemma}\label{lema4}
Let $G$ be a semisimple group admitting a group of automorphisms $A$ such that $|\mathcal{E}_G(a)|\leq m$ for all $a \in A$. Then the order of $[G, A]$ is $m$-bounded.
\end{lemma}

\begin{proof}
Observe that $[G, A] = \prod_{a \in A} [G, a]$. Lemma \ref{G_semis_lim} tells us that the order of $[G, a]$ is $m$-bounded for every $a\in A$. Therefore, it remains to show that $|A|$ is $m$-bounded. Since $|A|$ is bounded in terms of the orders of its abelian subgroups, without loss of generality we can assume that $A$ is abelian. Choose $a\in A$ such that $[G,a]$ is of maximal order. If $A$ acts faithfully on $[G,a]$, then the order of $A$ is $m$-bounded. Otherwise, choose a nontrivial automorphism $b\in A$ such that $[G,a,b] = 1$. Then, as mentioned earlier, $[G,b]\cap[G,a]=1$ and $[G,b,a] = 1$. We have $[G, b][G, a]\leq[G, ab]$, which contradicts the choice of $a$.
\end{proof}
It is well-known that any finite group $G$ has a normal series
$$
1=G_0\leq G_1\leq \cdots \leq G_{2h+1}=G
$$
such that $G_{i}/G_{i-1}$ is soluble (possibly trivial) if $i$ is odd and a direct product of non-abelian simple groups if $i$ is even. The minimal number $\lambda(G)$ of insoluble factors in such a series is called the insoluble length of $G$ \cite{ks_non}. As usual, the symbol $F^*(G)$ denotes the generalised Fitting subgroup of $G$.

\begin{lemma}\label{lim_com_n_sol}
Let $G$ be a group admitting an automorphism $a$ such that $|\mathcal{E}_G(a)| \leq m$ and $G=[G,a]$. Then $\lambda(G) \leq m-1$.
\end{lemma}

\begin{proof} We will prove this by induction on $m$. Let $R = R(G)$ be the soluble radical of $G$. Passing to the quotient $G/R(G)$ without loss of generality we assume that $R=1$. Write $F^*({G}) = S_1 \times \cdots \times S_k$, where $S_i$ are non-abelian simple groups.
	
Taking into account that $G=[G,a]$, note that $\mathcal{E}(a) \cap F^*({G}) \neq 1$ since $F^*({G})$ is not nilpotent. Thus, $|\mathcal{E}_{{G}/F^*({G})}(a)| \leq m - 1$. By induction, we conclude that $\lambda({G}/F^*({G})) \leq m - 2$. Since $\lambda(F^*({G})) = 1$, it follows that $\lambda(G) \leq m-1$.
\end{proof}

The following observation is almost obvious. We include a proof for the reader's convenience.
\begin{lemma}\label{op_quo_lim} Let $G$ be a finite group such that $O_p(G) = 1$, and let $N$ be a normal subgroup of $G$. Then the order of $O_p(G/N)$ is bounded in terms of the order of $N$ only.
\end{lemma}

\begin{proof} Let $L/N = O_p(G/N)$, and let $P_0$  be a Sylow $p$-subgroup of $L$. Obviously, $L = NP_0$. Set $P_1 = C_{P_0}(N)$. Thus, the index $[P_0:P_1]$ is bounded in terms of the order of $N$ only. Note that $NP_1$ is normal in $G$, whence $P_1\leq O_p(NP_1) \leq O_p(G) = 1$. Therefore $|P_0|$ is $|N|$-bounded, as required.
\end{proof} 

For a fixed prime $p$,  whenever $N$ is normal subgroup of a finite group $G$ such that $O_p(G)=1$, we denote by $\widehat{N}$ the full preimage of $O_p(G/N)$. It follows from the previous lemma that $|\widehat{N}|$ is $|N|$-bounded.

\begin{lemma}\label{[G:O_p]_m-lim}
	Let $G$ be a finite soluble group and $P$ a Sylow $p$-subgroup of $G$. If $|P| \leq m$, then the index $[G : O_{p'}(G)]$ is at most $m!$.
\end{lemma} 
\begin{proof}
We may assume that $O_{p'}(G) = 1$. Let $M = O_p(G)$. Since $C_G(M) \leq M$, it follows that $G/C_G(M)$ is isomorphic to a subgroup of $Aut(M)$. Obviously $|Aut(M)|\leq (m-1)!$ so we deduce that $|G|\leq m!$.
\end{proof} 

\begin{lemma}\label{[G:O_p,p']_m_lim} Let $G$ be a finite soluble group and $P$ a Sylow $p$-subgroup of $G$ such that $|\mathscr{E}(g)| \leq m$ for all $g\in P$. Then the index of $O_{p,p'}(G)$ in $G$ is $m$-bounded.
\end{lemma}
\begin{proof} Without loss of generality, we may assume that $O_p(G) = 1$. Set $N = O_{p'}(G)$ and observe that $C_G(N)\leq N$. In particular, the action of $P$ on $N$ is faithful. By Lemma \ref{[N,P]_m-lim}, the order of $[N,P]$ is $m$-bounded. 
	
Therefore, $P$ has $m$-bounded order, which, in view of Lemma \ref{[G:O_p]_m-lim}, implies that $N$ has $m$-bounded index in $G$.
\end{proof}

We are ready to prove Theorem \ref{main1}, which we restate here for the reader's convenience.\medskip

\noindent {\it Let $p$ be a prime and $G$ a finite group such that $|\mathscr{E}(g)| \leq m$ for all $p$-elements $g\in G$. Then $G$ has a normal subgroup $N$ such that $G/N$ is a $p'$-group and the index $[N:O_p(G)]$ is bounded in terms of $m$ only.}

\begin{proof} Without loss of generality, we may assume that $O_p(G)=1$. Choose a Sylow $p$-subgroup $P\leq G$ and set $N=\langle P^G\rangle$. It is sufficient to show that the order of $N$ is $m$-bounded. Since $N=\langle P^N\rangle$, without loss of generality we may assume that $G=N$. We now need to show that $G$ has $m$-bounded order. 
	
Suppose first that $G$ is soluble and set $M=O_{p'}(G)$. By Lemma \ref{[G:O_p,p']_m_lim}, the subgroup $M$ has $m$-bounded index in $G$. Therefore the order of $P$ is $m$-bounded. Let $P=\{g_1,\ldots,g_s\}$. Since $[M,P]=\prod_{i=1}^{s}[M,g_i]$, in view of Lemma \ref{[N,P]_m-lim} the order of $[M,P]$ is $m$-bounded. Thus, $\langle[M,P]^G\rangle$ has $m$-bounded order.
	
Let $\widehat{L}$ be the full preimage of $O_p(G/\langle[M,P]^G\rangle)$. By virtue of Lemma \ref{op_quo_lim}, the order of $\widehat{L}$ is $m$-bounded. Therefore we can pass to the quotient $G/\widehat{L}$. Observe that the image of $M$ in $G/\widehat{L}$ is central. Hence, $Z(G/\widehat{L})$ has $m$-bounded index. Therefore, by Schur's Theorem \cite[Theorem 10.1.4]{robinson}, we conclude that the commutator subgroup of $G/\widehat{L}$ has $m$-bounded order. We can now pass to the quotient $G/G'$ and assume that $G$ is abelian. Keeping in mind that $G=\langle P^G\rangle$ we deduce that $G$ has $m$-bounded order, as required.

We now drop the assumption that $G$ is soluble. Since $G=\langle P^G\rangle$, we infer from Lemma \ref{lim_com_n_sol} that $\lambda(G)$ is $m$-bounded. We will now use induction on $\lambda(G)$.

Write $R=R(G)$ for the soluble radical of $G$ and suppose first that $R=1$. Let $S=F^*(G)$. By Lemma \ref{lema4}, the order of $[S,P]$ is $m$-bounded. By induction, also the index of $S$ in $G$ is $m$-bounded. It follows that the order of $\langle [S,P]^G\rangle$ is $m$-bounded as well.
	
Observe that $S=\langle[S,P]^G\rangle$. Indeed, the image of $S$ in $G/\langle[S,P]^G\rangle$ is central. Since $S$ is a product of non-abelian simple groups, it follows that $S=\langle[S,P]^G\rangle$.
	
Thus, the induction on $\lambda(G)$ shows that in the case where $R=1$ the order of $G$ is $m$-bounded.
	
We therefore conclude that $R$ has $m$-bounded index in $G$. We know that the theorem holds for soluble groups. Applying the theorem to the subgroup $RP$, we deduce that the order of $[R,P]$ is $m$-bounded. Hence, so is the order of $\langle[R,P]^G\rangle$. Let $N$ be the full preimage in $G$ of $O_p(G/\langle[R,P]^G\rangle)$. By Lemma \ref{op_quo_lim} the order of $N$ is $m$-bounded. The image of $R$ in the quotient $G/N$ is central. It follows from Schur's Theorem that the order of $G'$ is $m$-bounded. Thus, passing to the quotient $G/G'$, we can assume that $G$ is abelian, and the result is immediate.
\end{proof}

\section{Profinite groups}

In this section we prove Theorem \ref{main2}. In what follows, unless stated otherwise, a subgroup of a profinite group will always mean a closed subgroup, all homomorphisms will be continuous, and quotients will be by closed normal subgroups. We also say that a subgroup is generated by a subset $X$ if it is generated by $X$ as a topological group.

We consider profinite groups $G$ with finite Engel sinks of $p$-elements. Whenever $g$ is a $p$-element of such a profinite group there is the smallest Engel sink of $g$, which, just as in the case of finite groups, we denote by $\mathscr{E}(g)$.

We write $FC(G)$ for the $FC$-centre of a group $G$, that is, the set of all elements $g\in G$ such that $|g^G|<\infty$. Note that if $G$ is a profinite group, then $FC(G)$ need not be closed so in this case we treat $FC(G)$ as merely a subset of $G$ rather than a subgroup. An element of $G$ lying in $FC(G)$ is called an $FC$-element. A useful fact, known as Dicman's lemma, is that if $x$ is an $FC$-element of finite order, then the normal subgroup $\langle x^G\rangle$ is finite \cite[Theorem 14.5.7]{robinson}. 

We write $F(G)$ for the Fitting subgroup, that is, the maximal normal pronilpotent subgroup of a profinite group $G$. We use $O_p(G)$ to denote the maximal normal pro-$p$ subgroup of $G$.

We will require the following theorem, which is immediate from \cite[Theorem 1.3]{as}.

\begin{theorem}\label{as} Let $p$ be a prime and $G$ a profinite group in which all $p$-elements are $FC$. Then $G$ has an open subgroup of the form $P\times Q$, where $P$ is an abelian pro-$p$ subgroup and $Q$ is a pro-$p'$ subgroup.
\end{theorem}

We will now prove Theorem \ref{main2}.\medskip

\noindent {\it Let $p$ be a prime and $G$ a profinite group in which every $p$-element possesses a finite Engel sink. Then $G$ has a normal subgroup $N$ such that $N$ is virtually pro-$p$ while $G/N$ is a pro-$p'$ group.}
\begin{proof}
For any $p$-element $g\in G$ let $N_g$ denote the maximal open normal subgroup of $G$ whose intersection with $\mathscr{E}(g)$ is trivial. Observe that whenever $x\in N_g$ there is a positive integer $i$ such that $[x,{}_ig]=1$. It follows that $g\in F(N_g\langle g\rangle)$ and, more precisely, $g\in O_p(N_g\langle g\rangle)$. We conclude that $[N_g,g]$ is a subnormal pro-$p$ subgroup of $G$ and in particular $[N_g,g]\leq O_p(G)$.

Hence, all $p$-elements of $G/O_p(G)$ are $FC$. Since $G/O_p(G)$ does not contain nontrivial normal pro-$p$ subgroups, it follows from Theorem \ref{as} that $G/O_p(G)$ is virtually pro-$p'$ and so Sylow $p$-subgroups of $G/O_p(G)$ are finite. Taking into account that Sylow $p$-subgroups of $G/O_p(G)$ consist of $FC$-elements and using Dicman's lemma, we deduce that if $P$ is a Sylow $p$-subgroup of $G$, then the image of $\langle P^G\rangle$ in $G/O_p(G)$ is finite. Therefore the subgroup $N=\langle P^G\rangle$ has the required properties. The proof is complete.
\end{proof}

\end{document}